\documentclass[12pt, reqno,a4paper, twoside]{amsproc}
\usepackage{ifpdf}
\usepackage[english]{babel}
\usepackage{amsmath}
\usepackage{amsthm, amssymb,bm}
\usepackage{geometry}
\usepackage{fullpage}
\usepackage{ucs}
\usepackage{tikz}
\usepackage{hyperref}
\usetikzlibrary{matrix}
\usepackage{mathrsfs}
\usepackage{eucal}
\hypersetup{
  colorlinks=true,
  citecolor=black,
  linkcolor=black,
  urlcolor=black,
  filecolor=red}

\usepackage{aliascnt}
\numberwithin{equation}{section}

\newtheorem{mainthm}{Theorem}

\newtheorem{theo}{Theorem}[section]	
\newtheorem*{theo*}{Theorem}

\newaliascnt{lem}{theo}
\newtheorem{lem}[lem]{Lemma}
\aliascntresetthe{lem}

\newaliascnt{propo}{theo}
\newtheorem{propo}[propo]{Proposition}
\aliascntresetthe{propo}

\newaliascnt{corol}{theo}
\newtheorem{corol}[corol]{Corollary}
\aliascntresetthe{corol}

\newaliascnt{ques}{theo}

\aliascntresetthe{ques}

\newaliascnt{conj}{theo}

\aliascntresetthe{conj}

\newaliascnt{assumption}{theo}

\aliascntresetthe{assumption}

\theoremstyle{remark}

\newaliascnt{rem}{theo}

\aliascntresetthe{rem}

\newaliascnt{exam}{theo}

\aliascntresetthe{exam}

\theoremstyle{definition}
\newaliascnt{defi}{theo}

\aliascntresetthe{defi}

\newaliascnt{nota}{theo}

\aliascntresetthe{nota}

\addto\extrasenglish{%
}

\DeclareMathOperator{\irr}{Irr}

\DeclareMathOperator{\aut}{Aut}
\newcommand{\Span}{\operatorname{Span}}
\DeclareMathOperator{\spec}{Spec}

\renewcommand{\hom}{\mathrm{Hom}}
\newcommand{\res}{\mathrm{Res}}
\newcommand{\ind}{\mathrm{Ind}}

\newcommand{\Tr}{\mathrm{Tr}}

\newcommand{\id}{\boldsymbol{1}}
\newcommand{\Ql}{\underline{\widebar{\dbQ}_\ell}}

\newcommand{\Cen}{\ensuremath{\mathrm{C}}}

\renewcommand{\L}{\mcal{L}}

\newcommand{\dbN}{\mathbb N}
\newcommand{\dbZ}{\mathbb Z}
\newcommand{\dbF}{\mathbb F}
\newcommand{\dbQ}{\mathbb Q}
\newcommand{\dbR}{\mathbb R}
\newcommand{\dbC}{\mathbb C}

\newcommand{\kk}{k}
\newcommand{\frob}{\sigma}

\newcommand{\gen}[1]{\langle{#1}\rangle}
\newcommand{\set}[1]{\left\{{#1}\right\}}

\newcommand{\abs}[1]{\left|#1\right|}
\newcommand{\inner}[1]{\left(#1\right)}

\newcommand{\mcal}{\mathcal}
\newcommand{\mbf}{\mathbf}
\newcommand{\mfr}{\mathfrak}
\newcommand{\msf}{\mathsf}

\newcommand{\widebar}{\overline}
\renewcommand{\tilde}{\widetilde}
\newcommand{\family}{\mathbf{F}}
\newcommand{\mset}{\mfr{M}}

\title{Bounds on multiplicities of spherical spaces over finite fields - the general case}
\author{Shai Shechter}
\email{shai.shechter@weizmann.ac.il}
\address{Faculty of Mathematics and Computer Science,
	The Weizmann Institute of Science,
	234 Herzl Street, Rehovot 7610001 Israel}

\begin{document}
\begin{abstract}
	Let $G$ be a connected reductive group scheme acting on a spherical scheme $X$. In the case where $G$ is of type $A_n$, Aizenbud and Avni proved the existence of a number $C$ such that the multiplicity $\dim\hom(\rho,\dbC[X(F)])$  is bounded by $C$, for any finite field $F$ and any irreducible representation $\rho$ of $G(F)$. In this paper, we generalize this result to the case where $G$ is a connected reductive group scheme over $\dbZ$, and prove Conjecture~A of \cite{AizenbudAvni-Bounds}.
\end{abstract}
\maketitle

\section{Introduction}

Let $G$ be a reductive group over a field $\kk$. A normal $G$-variety $X$ is called spherical if for any Borel subgroup $B\subseteq G$, there is a dense $B$-orbit in $X$. In this paper we prove the following.

\begin{mainthm}[Aizenbud-Avni~{\cite[Conjecture~A]{AizenbudAvni-Bounds}}]\label{conj:aizenbud-avni} 
	Let $G$ be a smooth reductive group scheme over $\dbZ$ acting on a scheme $X$. Assume that, for every prime $p$, the variety $X_{\dbF_p}=X\times\spec(\dbF_p)$ is $G_{\dbF_p}$-spherical. Then there is a constant $C$ such that, for any finite field $F$ and every irreducible representation $\rho$ of $G(F)$, the multiplicity of $\rho$ in $\dbC[X(F)]$ is less than $C$.
\end{mainthm}

From here on, we assume $\kk=\dbF_q$ is a finite field of characteristic $p$. We write $\frob$ for the Frobenius map associated to the $\kk$-structure of $G$, i.e. such that $G(\kk)=G^\frob$. 

Our proof is based on the observation that certain key aspects of \cite[Theorem~B]{AizenbudAvni-Bounds} and its proof carry over to the generality of $G$ a smooth reductive $\dbZ$-group scheme and $\rho$ a (possibly virtual) representation, afforded by a Deligne-Lusztig character, i.e.\ a character of the form $R_T^G(\theta)$, where $T$ is a $\kk$-defined maximal torus of $G$, obtained from an irreducible character $\theta$ of $T(\kk)$ by means of Deligne-Lusztig induction (see \cite{DeligneLusztig}). More specifically, we show that the constant obtained in \cite[Theorem~B]{AizenbudAvni-Bounds} may in fact serve as a bound for the value $\inner{R_T^G(\theta),\chi_{\dbC[X(\kk)]}}_{G(\kk)}$, where  $\chi_{\dbC[X(\kk)]}$ denotes the character of $\dbC[X(\kk)]$; see \autoref{propo:bounds-induced-chars}. 

Following this, we invoke certain properties of Lusztig's non-abelian Fourier transform for families of unipotent characters and Jordan Decomposition Theorem \cite[Theorem~4.23]{Lusztig-Book}, to show that in fact all irreducible characters of $G(\kk)$ occur as components in a virtual character of the form $\sum_{i=1}^r\alpha_iR_{T_i}^G(\theta_i)$, with $\abs{\alpha_i}$ and $r$ bounded independently of $\kk$. Furthermore, the coefficient of each irreducible character in this specific virtual character is taken from a fixed finite set of positive real numbers, determined by the root system of $G$. The boundedness assertion for irreducible representations then follows from the bound described in the previous paragraph.

A rather more conceptual proof of the boundedness statement of \autoref{conj:aizenbud-avni} may possibly be obtained by an intrinsic analysis the \emph{cuspidal} character sheaves of $G$ and their associated multiplicity complex. This would be the center of attention of an upcoming version of this paper.

\subsection{Notation and Organization}
In \autoref{subsection:preliminaries} we recall the and extend the proof of \cite[Theorem~2.13]{AizenbudAvni-Bounds} to the generality of Delinge-Lusztig characters of arbitrary smooth reductive group schemes (see \autoref{propo:bounds-induced-chars}). Consequently, in \autoref{section:bounds}, we prove the main statement of this paper, initially for the case of unipotent characters (\autoref{subsection:bounds-unipotent}, and subsequently, in the general setting (\autoref{subsection:bounds-arbitrary}). 

Throughout, given a finite group $\Gamma$, we write $\irr(\Gamma)$ to denote the set of irreducible complex valued characters of $\Gamma$ and denote by $\inner{\cdot,\cdot}_\Gamma$ the standard $\Gamma$-invariant inner product on the space of class functions of $\Gamma$. We fix $\ell\ne \mathrm{Char}(\kk)$, and an isomorphism $\widebar{\dbQ}_\ell\simeq \dbC$. 

As a matter of terminology, given an irreducible character $\chi$ and a virtual character $R$, we say that $\chi$ is a \textit{component} of $R$ if $R$ projects non-trivially onto the span of $\chi$. We reserve the term \textit{constituent} for the case where $R$ is a genuine character, and $\chi$ is a direct summand of $R$.

%
%
\section{Bounds on multiplicities of induced character sheaves}\label{subsection:preliminaries}
The proof of the following proposition follows the same path as \cite[Theorem~2.1.3]{AizenbudAvni-Bounds}; for completeness, we record it here, noting the necessary adjustments. As in loc.\ cit.\ , given an algebraic group $G$ and a $G$-variety $Y$, we write $Y_G=\set{(g,y)\in G\times Y: gy=y}$. Assuming further that $G$ is reductive with Borel subgroup $B$, we write $c(Y,G)$ for the number of irreducible components of $(Y\times G/B)_G$. The space of complex-valued functions on $Y(\kk)$ is denoted $\dbC[Y(\kk)]$ and $\chi_{\dbC[Y(\kk)]}$ is the character of the permutation representation of $G(\kk)$ on this space. We prove the following.


\begin{propo} \label{propo:bounds-induced-chars} Let $G$ be a connected reductive group over a finite field $\kk$, and let $X$ be a $\kk$-defined spherical $G$-space. Then for any $\kk$-defined maximal torus $T$ and any character $\theta\in\irr(T(\kk))$, \[\abs{\inner{R_T^G(\theta),\chi_{\dbC[X(\kk)]}}_{G(\kk)}}\le  c(G,X).\]
\end{propo}

The proof of \autoref{propo:bounds-induced-chars} is modelled after \cite[Theorem~2.2.2]{AizenbudAvni-Bounds}. The key component of the proof is to show that a specific character sheaf, whose associated function is (up to a constant) the virtual character $R_T^G(\theta)\cdot \chi_{(\dbC[X(\kk)])}=R_T^G(\theta)\cdot \widebar{\chi_{(\dbC[X(\kk)])}}$, gives rise to a mixed and bounded cohomological complex of weight $\le 0$. 

Let us recall the\textit{ multiplicity complex} introduced in \cite[\S~2.3]{AizenbudAvni-Bounds}. Fix $T$ and $\theta$, as in the proposition, and let $B\supseteq T$ be a Borel subgroup of $G$. The coset space $G/B$ is naturally identified with the variety of Borel subgroups of $G$ via $hB\mapsto {}^hB=hBh^{-1}$. 

Let $\mcal{L}=\L_\theta$ denote the Kummer local system (over $\widebar{\dbQ}_\ell$) on $T$ determined by $\theta$. Consider the diagram in \autoref{diagram:multiplicity-complex}, 
\begin{figure}[h]
\begin{tikzpicture}
  \matrix (m) [matrix of math nodes,row sep=2em,column sep=3em,minimum width=2em]
  {
		&(X\times G/B)_G\\
	X_G	&	&(G/B)_G&V\\
		&G&&T\\
		&\text{pt}
      \\};
  \path[-stealth]
    (m-1-2) edge node [above] {$\tilde{\pi}$} 	(m-2-1)
    (m-2-1) edge node [below] {$f$} 			(m-3-2)
    (m-2-4) edge node [above] {$\tau_2$} 		(m-2-3)
    (m-2-4) edge node [right] {$\tau_1$} 		(m-3-4)
	(m-1-2) edge node [above] {$\tilde{f}$}		(m-2-3)
    (m-2-3)	edge node [below] {$\pi$}			(m-3-2)
    (m-1-2) edge node [right] {$p$} 			(m-3-2)
    (m-3-2) edge node [right] {$q$}				(m-4-2);
\end{tikzpicture}\caption{}\label{diagram:multiplicity-complex}
\end{figure}
in which $f,\pi,\tilde{f},\tilde{\pi}$ and $p$ are the coordinate projection maps, $V:=\set{(g,h)\in G\times G: h^{-1}gh\in B}$,
$\tau_1(g,h)=\pi_B(h^{-1}gh)$, 
where $\pi_B:B\to B/[B,B]\simeq T$ is the natural quotient map, and $\tau_2(g,h)=(g,{}^hB)$.

There exists a one dimensional local system $\mcal{K}$ on $(G/B)_G$ such that $\tau_2^*\mcal{K}\simeq\tau_1^*\mcal{L}$; see \cite[\S~4]{Lusztig-CSI}. Given a character sheaf $\mcal{M}$ on $G$ with pure Weil structure $\alpha:\frob^*\mcal{M}\to\mcal{M}$ of weight zero we write $\chi_{\mcal{M},\alpha}$ for the associated function on $G(\kk)$, i.e. $\chi_{\mcal{M},\alpha}(g)=\Tr(\alpha_g,\mcal{M}_g)$. If the Weil structure $\alpha$ is clear from context, we occasionally omit it. The relation between the character sheaf $\mcal{K}$ and the Deligne-Lusztig character induced from $\theta$ is described explicitly by the following theorem of Laumon.
\begin{theo}[{\cite[Corollary~2.3.2]{Laumon}}]\label{theo:laumon} Let $\mcal{L}=\mcal{L}_\theta$ be the local system associated to $\theta\in\irr(T(\kk))$ and let $\mcal{K}$ be as above. Then
\[\chi_{R\pi_!\mcal{K}}=(-1)^{\dim T}R_T^G(\theta).\]
\end{theo}


\begin{lem}
\label{lem:inner-product}
Let $\mcal{M}$ be a character sheaf on $G$ with pure Weil structure $\alpha$. Let $\beta$ be the natural Weil structure on $Rq_!(\mcal{M}\otimes^L Rf_!\Ql)$. Then 
\[\chi_{Rq_!(\mcal{M}\otimes^L Rf_!\Ql),\beta}=\abs{G(\kk)}\inner{\chi_{\mcal{M},\alpha},\chi_{\dbC[X(\kk)]}}_{G(\kk)}.\]
Furthermore, if $\mcal{M}=R\pi_!\mcal{K}$ and $\alpha$ is its natural Weil structure, then $\inner{\chi_{\mcal{M},\alpha},\vartheta}_{G(\kk)}\in\dbZ$ for any character $\vartheta$ of $G(\kk)$. 
\end{lem}
\begin{proof}
For the first assertion, see {\cite[Lemma~2.3.2]{AizenbudAvni-Bounds}}. To prove the second assertion, note that by \autoref{theo:laumon}, we may write 
\[\chi_{\mcal{M},\alpha}=\chi_{R\pi_!\mcal{K}}=(-1)^{\dim T}\sum_{i\in\dbZ}(-1)^i \vartheta_i,\]
where $\vartheta_i$ is the character of the action of $G(\kk)$ on the $\Ql$-vector space $H^i_c(X,\Ql)$, where $X=\tilde{X}(\dot{w})$ is as in \cite[Definition~1.9]{DeligneLusztig}. Since the $\vartheta_i$ are characters of genuine representations of $G(\kk)$, we have that $\inner{\vartheta_i,\vartheta}\in\dbZ_{\ge 0}$ for all $i\in\dbZ$. The second assertion follows, since $H^i_c(X,\Ql)=0$ for all but finitely many $i$'s. 
\end{proof}

The following lemma gives an alternative description of the sheaf considered in \autoref{lem:inner-product}, as the push-forward of a single one-dimensional local system on $(X\times G/B)_G$.
\begin{lem}\label{lem:reformulate-tensor} $R\pi_!\mcal{K}\otimes^L Rf_!\Ql=Rp_!\tilde{\mcal{K}}$, where $\tilde{\mcal{K}}=\tilde{f}^*\mcal{K}$. 
\end{lem}
\begin{proof}
{See \cite[Lemma~2.3.3]{AizenbudAvni-Bounds}}.
\end{proof}

Finally, we have the following well-known fact.

\begin{lem}[{See \cite[Lemma~2.3.4]{AizenbudAvni-Bounds}}]\label{lem:bounds-on-cohomology-general}
Let $X$ be a $d$-dimensional variety and let $\mcal{F}$ be a one-dimensional local system on $X$. Then $H^i_c(X,\mcal{F})=0$ unless $i\in\set{0,\ldots,2d}$ and $\dim H_c^i(X,\mcal{F})$ is bounded by the number of irreducible components of $X$. 
\end{lem}

From the last two lemmas we obtain the following.
\begin{corol}\label{corol:mixed-bounded-complex} $Rq_!(R\pi_!\mcal{K}\otimes^L Rf_!\Ql)$ is a mixed complex of weight $\le 0$, concentrated in degrees $0,\ldots,2\dim G$ and $\dim H^{2\dim G}(Rq!(R\pi_!\mcal{K}\otimes^L Rf_!\Ql))\le c(X,G)$. 
\end{corol}
\begin{proof}
By \autoref{lem:reformulate-tensor}, it is enough to bound the weights and cohomologies of $R(q\circ p)_! \tilde{\mcal{K}}$. The boundedness of weights follow from the fact that push-forward with proper support maps mixed sheaves of weight $\le 0$ to mixed sheaves of weight $\le 0$. The boundedness of cohomologies follows from \autoref{lem:bounds-on-cohomology-general}, applied to the one-dimensional local system $\tilde{\mcal{K}}$ on $(X\times G/B)_G$.
\end{proof}

%
\begin{proof}[Proof of \autoref{propo:bounds-induced-chars}]
Writing $\beta_n$ for the natural Weil structure on $\mcal{A}:=R_{q!}(R\pi_!\mcal{K}\otimes^L R_{f!}\Ql)$ with respect to the Weil structure given on $R\pi_!\mcal{K}$ from $\frob^n$, and denoting $\kk_n=\dbF_{q^n}$, by \autoref{lem:inner-product} we have that 
\[\varphi(n)=\frac{\Tr(\beta^n,\mcal{A})}{\abs{G(\kk_n)}}=\inner{\chi_{\mcal{M},\alpha^n},\chi_{\dbC[X(\kk_n)]}}_{G(\kk_n)}.\]
By the second assertion of \autoref{lem:inner-product}, the RHS is an integer. Applying the same argument as \cite[Theorem~2.1.3]{AizenbudAvni-Bounds}, we have that $\varphi$ is a periodic function and asymptotically bounded by $c(X,G)$. By \autoref{lem:inner-product} It follows that
\[\abs{\varphi(1)}=\inner{\chi_{\mcal{M},\alpha},\chi_{\dbC[X(\kk)]}}_{G(\kk)}\le \limsup_{n\to\infty}\abs{\varphi(n)}\le c(X,G).\]
\end{proof}

\section{Bounds on multiplicities of irreducible characters}\label{section:bounds}
\subsection{Unipotent characters}\label{subsection:bounds-unipotent}
Let $G$ be a connected reductive group over $\kk$. Recall that an irreducible character $\chi$ of $G(\kk)$ is said to be \textit{unipotent} if it occurs with non-zero coefficient in the Deligne-Lusztig $R_T^G(\id_T)$, where $T$ is a $\kk$-defined maximal torus of $G$ (not necessarily within a $\kk$-defined Borel subgroup) and $\id_T$ is its trivial character. The set of unipotent characters is denoted $\mcal{E}(G(\kk),(1))$. An irreducible representation $\rho$ of $G(\kk)$ is said to be unipotent if its character is such.

Let $T_1\subseteq G$ be a fixed $\kk$-defined maximal torus within a $\kk$-defined Borel subgroup $B_1$, and put $W=W_G(T_1)$ and $\Phi=\Phi_G(T_1)$ for the corresponding Weyl group and root system, respectively. 

\begin{propo}\label{propo:bound-for-unipotent-chars}
	There exist numbers $M_\Phi,N_\Phi\in\dbN$, determined by the root system of $G$, such that, if $\abs{\kk}>N_\Phi$, then for any $\kk$-defined spherical $G$-space $X$ and any \emph{unipotent} representation $\rho$
\[\dim\hom(\rho,\dbC[X(\kk)])<M_\Phi\cdot\abs{W}\cdot c(G,X).\]
\end{propo}

\begin{corol}
	\autoref{conj:aizenbud-avni} holds for whenever $\rho$ is an irreducible unipotent representation.
\end{corol}

\begin{proof}[Proof of \autoref{propo:bound-for-unipotent-chars}]
	For any $w\in W$, let $B_w\subseteq G$ be a fixed Borel subgroup such that $B$ and $\sigma(B)$ are in relative position $w$, and let $T_w\subseteq B_w$ its maximal torus, with $T_1$ and $B_1$ as above. Write $R_w=R_{T_w}^G(\id_{T_w(\kk)})$. Let $\tau_\frob\in \aut(W)$ denote the map induced the action of $\sigma$ on $N_G(T)$. Any $\tau_\frob$-invariant irreducible character $\phi\in\irr(W)$ may be extended to a character of $\tilde{W}=W\ltimes\gen{\tau_\frob}$, and gives rise to an associated \textit{almost character} of $G(\kk)$, which is the virtual character defined by
		\[R_\phi=\frac{1}{\abs{W}}\sum_{w\in W}\phi(w\cdot \tau_\frob)R_w.\]
	The almost characters comprise, an alternarive, somewhat more approachable basis to the space spanned by the set $\set{R_w}_{w\in W}$.

	In a series of papers \cite{Lusztig-Exceptional,Lusztig-Orthogonal, Lusztig-Srinivasan,Lusztig-TypeE8, Lusztig-Symplectic} Lusztig et.\ al.\ analysed the decomposition of the almost characters as linear combinations of irreducible unipotent characters, for all simple adjoint groups (with few exceptions, complemented later on by Geck and Malle \cite{Geck-Malle}). The key statements relevant to this paper are summarised in \autoref{appendix:non-abelian-fourier}; specifically, see \autoref{lem:unipotents-pos-in-almost-char}. Concisely, one has the following.
	\begin{enumerate}
		\item The irreducible characters of $W$ are in natural bijection with the set of irreducible constituents of $R_1=\ind_{B_1(\kk)}^{G(\kk)}(\id_{B_1(\kk)})$. Given $\phi\in\irr(W)$ we write $\chi_\phi$ for the corresponding constituent of $R_1$. 
		\item The set of unipotent characters partitioned into \textit{families}: $\mcal{E}(G(\kk),(1))=\family_1\sqcup\ldots\sqcup\family_r$, and any family contains elements of the form $\chi_\phi$, for $\phi\in\irr(W)$. Given $\phi_1,\phi_2\in\irr(W)^{\tau_\frob}$, the characters $\chi_{\phi_1},\chi_{\phi_2}$ lie in the same family of $\mcal{E}(G(\kk),(1))$ if and only if $\phi_1$ and $\phi_2$ lie in the same family, in the sense of \cite[\S~4.2]{Lusztig-Book}.
		\item\label{item:3} 
		Assuming $\abs{\kk}\gg 0$, for any $\chi\subseteq\mcal{E}(G(\kk),(1))$, with $\family$ its family, there exists $\phi\in \irr(W)^{\tau_\frob}$ with $\chi_\phi\in \family$ such that the corresponding almost character $R_{\phi}$ is of the form $\sum_{\chi\in \family}\gamma_\chi\cdot \chi$ with $\gamma_\chi > 0$. Furthermore, there exists a finite set $C_\Phi\subseteq \dbR_{>0}$, determined by the root system of $G$ and not the ground field, such that all $\set{\gamma_{\chi}:\chi\in\family}\subseteq C_\Phi$. 
	\end{enumerate}
	Let $M_\Phi>\max\set{\gamma^{-1}:\gamma\in C_\Phi}$, and let $\chi\in\family$ be a unipotent character within its family. Let $\phi\in\irr(W)^{\tau_\frob}$ be the as in \autoref{item:3}. Then
		\[\abs{\inner{R_\phi,\chi_{\dbC[X(\kk)]}}_{G(\kk)}}\le\sum_{w\in W}\frac{\abs{\phi(w\cdot \tau_\frob)}}{\abs{W}}\abs{\inner{R_w,\chi_{\dbC[X(\kk)]}}_{G(\kk)}}\le \abs{W}\cdot c(G,X),\]
	where the last inequality follows from \autoref{propo:bounds-induced-chars} and the fact that $\abs{\phi(z)}\le\phi(1)\le \abs{W}$, for any $\phi\in\irr(W)$ and $z\in W$. As all components of $R_\phi$ are irreducible characters and occur with positive coefficient, we have that
		\[\abs{\inner{R_\phi,\chi_{\dbC[X(\kk)]}}_{G(\kk)}}= \sum_{\chi'\in\family}\gamma_{\chi'}\inner{\chi',\chi_{\dbC[X(\kk)]}}_{G(\kk)}\ge \frac{1}{M_\Phi}\inner{\chi,\chi_{\dbC[X(\kk)]}}_{G(\kk)}.\]
	The proposition follows.
\end{proof}

\subsection{Arbitrary irreducible characters}\label{subsection:bounds-arbitrary}
As above, let $G$ be a connected reductive group over $\kk$ and $T$ a $\kk$-defined maximal torus within a $\kk$-defined Borel subgroup $B$. Two Deligne-Lusztig characters $R_{T_1}^G(\theta_1),R_{T_2}^G(\theta_2)$, with $T_1,T_2$ $\kk$-defined maximal tori and $\theta_i\in\irr(T_i(\kk))$ have a common irreducible component if and only if the pairs $(T_1,\theta_1)$ and $(T_2,\theta_2)$ are \textit{geometrically conjugate}; see \cite[Definition~13.2]{DigneMichel}. 

The geometric conjugacy classes of $G(\kk)$ affords an elegant description in terms involving the algebraic dual group. Let $G^*$ denote the algebraic dual group of $G$, with respect to a maximal torus $T^*$, dual to $T$. The $G(\kk)$-conjugacy classes of pairs $(T_1,\theta_1)$ as above, are in bijection with the $G^*(\kk)$-conjugacy classes of pairs $(T^*_1,s_1)$, where $T^*_1\subseteq G^*$ is a $\kk$-defined torus, and $s_1\in T^*_1(\kk)$; see \cite[Proposition~13.13]{DigneMichel}. Under this bijection, two pairs $(T_1,\theta_1),(T_2,\theta_2)$ lie in the same \textit{geometric} conjugacy classes, if and only if the corresponding pairs $(T_i^*,s_i)$ ($i=1,2$) are such that $s_1$ and $s_2$ are $G^*(\kk)$-conjugate. 

If the $G(\kk)$-class of $(T_1,\theta_1)$ is mapped by this bijection to the $G^*(\kk)$-class of $(T^*_1,s_1)$, it is often convenient to write $R_{T_1^*}(s_1)$ for $R_{T_1}(\theta_1)$; note that this notation is well-defined. 

Given $(s)\subseteq G^*(\kk)$ a semisimple conjugacy class, one defines the associated \textit{Lusztig series} $\mcal{E}(G(\kk),(s))$ to be the set of irreducible characters of $G(\kk)$ which occur as components in a Deligne-Lusztig character $R_{T_1}(\theta_1)$, with $(T_1,\theta_1)$ in the geometric conjugacy class associated to $(s)$. One has
\begin{equation}\label{equation:partition-to-families}
\irr(G(\kk))=\bigsqcup_{(s)\in G^*\backslash\!\backslash G^*_{\rm ss}}\mcal{E}(G(\kk),(s));
\end{equation}
see \cite[Proposition~13.17]{DigneMichel}. Moreover, we have the following Jordan-decomposition theorem, proved in \cite[Theorem~4.23]{Lusztig-Book} assuming $G$ has connected center, and, for general connected reductive group, in \cite{Lusztig-Disconnected,Lusztig-Book}; see also \cite[Theorem~13.23]{DigneMichel}.
\begin{theo}[{Lusztig}]\label{theo:jordan-decomposition}
	Let $s\in G^*(\kk)$ be semisimple. There exists a bijection 		\[\psi_s:\mcal{E}(G(\kk),(s))\xrightarrow{1-1}\mcal{E}(\Cen_{G^*(\kk)}(s),(1)).\]\newcommand{\fqrk}{\kk\textrm{-rank}}
	Furthermore, extending $\psi_s$ by linearity to virtual characters, we have\[\psi_s\left(R^G_{T_1^*}(s)\right)=\epsilon_G \epsilon_{\Cen_{G^*}(s)}R_{T_1^*}^{\Cen_{G^*}(s)}(\id_{T_1^*}),\]
	for any maximal torus $T_1^*\subseteq G^*$ containing $s$. Here $\epsilon_H=(-1)^{\fqrk H}$, for $H$ a reductive $\kk$-group.
\end{theo}

Note that, as $\psi_s$ maps irreducible characters to irreducible characters, its extension to a linear map $\Span_{\dbC}\mcal{E}(G(\kk))\to \Span_{\dbC}\mcal{E}(\Cen_{G^*(\kk)}(s),(q))$ is an isometry with respect to the standard invariant inner products on both spaces; that is
\begin{equation}\label{equation:psi-s-isometry}
\inner{\psi_s(\alpha),\psi_s(\beta)}_{\Cen_{G^*(\kk)}(s)}=\inner{\alpha,\beta}_{G(\kk)}\quad\text{ for any }\alpha,\beta\in \Span_{\dbC}\mcal{E}(G(\kk),(s)).
\end{equation}

Fix $s\in G^*(\kk)$ semisimple, with $(s)$ its $G^*(\kk)$-conjugacy class. Let $T^*\subseteq G^*$ be a $\kk$-defined maximal torus containing $s$. Let $W^*(s)$ and $\Phi^*(s)$ be the Weyl group and root system of $\Cen_{G^*}(s)$ with respect to $T^*$, and, for any $w^*\in W^*(s)$ and $\phi\in\irr(W^*(s))$, put $R_{w^*}^{\Cen_{G^*(\kk)}(s)}:=R_{T^*_{w^*}}^{\Cen_{G^*}(s)}(\id_{T_{w^*}(\kk)})$ and $R_{\phi}^{\Cen_{G^*}(s)}:=\abs{W^*(s)}^{-1}\sum_{w^*\in W^*(s)}\phi(w^*)R_{w^*}^{\Cen_{G^*(\kk)}(s)}$, as in \autoref{subsection:bounds-unipotent}. Write $\tilde{R_\phi}^{\Cen_{G^*(\kk)}(s)}=\psi_{s}^{-1}(R_\phi^{\Cen_{G^*(\kk)}(s)})$.

\begin{lem} \label{lem:bound-on-preimage-psi-s}$\abs{\inner{\tilde{R_\phi}^{\Cen_{G^*(\kk)}(s)},\chi_{\dbC[X(\kk)]}}_{G^*(\kk)}}\le \abs{W}\cdot c(G,X)$, for any $\phi\in \irr(W^*(s))$.
\end{lem}
\begin{proof}
By \autoref{theo:jordan-decomposition}, we have that
\[\tilde{R_\phi}^{\Cen_{G^*(\kk)}(s)}=\frac{1}{\abs{W^*(s)}}\sum_{w^*\in W^*(s)}\phi(w^*)\psi_s^{-1}(R_{w^*}^{\Cen_{G^*(\kk)}(s)})=\frac{\epsilon_G\epsilon_{\Cen_{G^*}(s)}}{\abs{W^*(s)}}\sum_{w^*\in W^*(s)}\phi(w^*)R_{T_{w^*}^*}(s).\]
Applying the argument of the previous section, we have that
\begin{align}
\abs{\inner{\tilde{R_\phi}^{\Cen_{G^*(\kk)}(s)},\chi_{\dbC[X(\kk)]}}_{G^*(\kk)}}&\le \sum_{w^*\in W^*(s)}\frac{\abs{\epsilon_G\epsilon_{\Cen_{G^*}(s)}\phi(w^*)}}{\abs{W^*(s)}}\cdot\abs{\inner{R_{T^*_{w^*}}(s),\chi_{\dbC[X(\kk)]}}_{G^*(\kk)}}\notag\\&\le \abs{W^*(s)}\cdot\max_{w^*\in W^*(s)}\set{\abs{\inner{R_{T^*_{w^*}}(s),\chi_{\dbC[X(\kk)]}}_{G^*(\kk)}}}\notag\\
&\le\abs{W}\cdot c(G,X),\label{equation:upper-bound-arbitrary}
\end{align}
where the final inequality follows from \autoref{propo:bounds-induced-chars} and since $W^*(s)\le W^*\simeq W$. 
\end{proof}

\begin{lem}\label{lem:bound-dep-on-s}
Let $s\in G^*(\kk)$ be semisimple and $\rho$ an irreducible representation of $G(\kk)$ whose character lies in $\mcal{E}(G^*(\kk),(s))$. Then, assuming $\abs{\kk}>N_{\Phi^*(s)}$,
\[\dim\hom(\rho,{\dbC[X(\kk)]})<M_{\Phi^*(s)}\abs{W}c(G,X),\] 
where $\Phi^*(s)$ is the root system of $\Cen_{G^*}(s)$ and $N_{\Phi^*(s)},M_{\Phi^*(s)}$ are as in \autoref{propo:bound-for-unipotent-chars}.
\end{lem}
\begin{proof}
Let $\chi_\rho$ denote the character of $\rho$ and consider the unipotent character $\vartheta_\rho=\psi_s(\chi_\rho)\in \mcal{E}(\Cen_{G^*(\kk)}(s),(1))$. As in the previous section, there exists $\phi\in \irr(W^*(s))$ such that $\psi_s(\chi_\rho)$ occurs with a positive coefficient in $\widetilde{R}_\phi^{\Cen_{G^*(\kk)}(s)}$, and such that $\widetilde{R}_\phi^{\Cen_{G^*(\kk)}(s)}=\sum_{\vartheta'\in \family}\gamma_{\vartheta'}{ \vartheta'}$, with $\family\subseteq\mcal{E}(\Cen_{G^*(\kk)}(s),(1))$ the family containing $\vartheta_\rho$. By linearity of $\psi_s$, we have that $\tilde{R_\phi}^{\Cen_{G^*(\kk)}(s)}=\psi_s^{-1}(R_\phi^{\Cen_{G^*(\kk)}(s)})=\sum_{\vartheta'\in \family}\gamma_{\vartheta'}\psi_s^{-1}(\vartheta')$ is a linear combinations of irreducible characters in $\mcal{E}(G(\kk),(s))$ with positive coefficients taken from the fixed finite set $C_{\Phi^*(s)}\subseteq\dbQ_{>0}$. Taking $M_{\Phi^*(s)}$ to be as in \autoref{propo:bound-for-unipotent-chars}, we have that 
\begin{align*}
	\inner{\widetilde{R}^{\Cen_{G^*(\kk)}(s)}_\phi,\chi_{\dbC[X(\kk)]}}_{G(\kk)} &=\gamma_\vartheta\inner{\chi_\rho,\chi_{\dbC[X(\kk)]}}_{G(\kk)}+\sum_{\vartheta\ne \vartheta'\in\family}\gamma_{\vartheta'}\inner{\psi_s^{-1}(\vartheta'),\chi_{\dbC[X(\kk)]}}_{G(\kk)}\\
	&\ge
	\frac{1}{M_{\Phi^*(s)}}\inner{\chi_\rho,\chi_{\dbC[X(\kk)]}}_{G(\kk)}.
\end{align*}
In particular
\[\dim\hom(\rho,{\dbC[X(\kk)]})\le M_{\Phi^*(s)}\abs{\inner{\tilde{R_\phi}^{\Cen_{G^*(\kk)}(s)},\chi_{\dbC[X(\kk)]}}_{G(\kk)}}\le M_{\Phi^*(s)}\abs{W}c(G,X),\] 
by \autoref{lem:bound-on-preimage-psi-s}, as wanted.
\end{proof}

\begin{proof}[Proof of \autoref{conj:aizenbud-avni}] Let $M=\max\set{M_\Sigma}$ and $N=\max\set{N_\Sigma}$, where $\Sigma$ ranges over the finite set \[\set{\Phi}\cup \set{\Psi^*<\Phi^*:\Psi^*\text{ a closed subsystem}}\] and put \[C=\max\left(\set{M\cdot \abs{W}\cdot c(G,X)}\cup\set{\dim\hom(\rho,\chi_{\dbC[X(F)]}):\abs{F}<N\text{ and }\rho\in \irr(G(\dbF_q))}\right).\] The theorem follows from \autoref{propo:bound-for-unipotent-chars} and \autoref{lem:bound-dep-on-s}.

\end{proof}

\appendix
\section{Families of unipotent characters and non-abelian Fourier transform}\label{appendix:non-abelian-fourier}
In this appendix we prove the following lemma, which is a consequence of Lusztig's description of the unipotent characters of simple adjoint groups.
\begin{lem}\label{lem:unipotents-pos-in-almost-char}
Let $\Phi$ be a crystallographic root system. There exists $N\in\dbN$ and a finite set $C_\Phi\subseteq \dbR_{>0}$ such that the following holds for any finite field $\kk$ with $\abs{\kk}>N$ and any $\kk$-defined reductive (not necessarily connected) group $G$ with absolute root system $\Phi$.

For any unipotent character $\chi\in \mcal{E}(G(\kk),(1))$ there exists a virtual character $R(\chi)$ of the form $R(\chi)=\sum_{w\in W}\alpha_w R_w$, with $\abs{\alpha_w}<1$ and $R_w=R_{T_w}^G(\id_{T_w(\kk)})$ Deligne-Lusztig characters, which may be written as
\[R(\chi)=\sum_{i=1}^r\gamma_i\chi_i,\]
with $\gamma_1,\ldots,\gamma_r\in C_\Phi,\: \chi_1,\ldots,\chi_r\in\mcal{E}(G(\kk),(1))$ and $\chi_1=\chi$. 
\end{lem}

\subsection{Reduction to simple adjoint groups}
Applying \cite[Proposition~13.20]{DigneMichel} initially for the inclusion $G^\circ\stackrel{\iota}{\hookrightarrow} G$ and then for the quotient map $G^\circ\xrightarrow{\pi} G^\circ_\mathrm{ad}=G^\circ/Z(G^\circ)$, we have that any unipotent character of $G$ is the unique extension of a unipotent character of the form $\chi\circ\pi\in\mcal{E}(G^\circ(\kk),(1))$ for a unique unipotent character $\chi$ of $G_\mathrm{ad}^\circ$. That is, we have a bijection $\mcal{E}(G(\kk),(1))\xrightarrow{1-1}\mcal{E}(G^\circ_{\mathrm{ad}}(\kk),(1))$, and it suffices to prove \autoref{lem:unipotents-pos-in-almost-char} for $G_\mathrm{ad}^\circ$ in order to obtain the general claim.

Assuming therefore that $G$ is semisimple and adjoint, we note that \autoref{lem:unipotents-pos-in-almost-char} is well-behaved with respect to finite direct products. Indeed, if $G=G_1\times G_2$, with $G_i$ a reductive $\kk$-defined group with root system $\Phi_i$, since any unipotent character of $G(\kk)$ is the tensor product of unipotent characters of $G_1(\kk)$ and $G_2(\kk)$, one easily verifies that the set $C_{\Phi_1\sqcup \Phi_2}$ may be taken to be $C_{\Phi_1}\cdot C_{\Phi_2}=\set{x_1x_2:x_i\in C_{\Phi_i},\:i=1,2}$. Therefore, we may assume that the Frobenius map of $G$ induces a cyclic permutation on the irreducible components of $\Phi$, and thus $G=\res_{K\mid\kk}H$ for a finite field extension $K$ of $\kk$ and $H$ a $K$-defined simple adjoint group. As $G(\kk)=H(K)$ and $\abs{K}\ge \abs{\kk}$, it would suffice to prove the lemma under the assumption the $G$ is simple and adjoint over $\kk$. 

\subsection{Split adjoint groups}

Let us first assume $G$ is split, and hence the Frobenius map $\frob$ fixes $W$ pointwise. For each $w\in W$, we fix a Borel subgroup $B_w\in G$ in relative position $w$ (i.e. such that $\frob(B)=w(B)$), and let $T_w$ be its maximal torus. To every such $w$, one associates a Deligne-Lusztig character $R_w=R_{T_w}^G(\id_{T_w(\kk)})$,  which is the character of the virtual representation $\sum_i (-1)^i H_c^i(X_w,\overline{\dbQ}_\ell)$ of $G(\kk)$, where $X_w\subseteq G/B$ is the variety of Borel subgroups of relative position $w$. The unipotent characters of $G(\kk)$ are precisely the irreducible components of the $R_w$; however, obtaining a description of  unipotent characters is more approachable via a different set of virtual characters. Given a character $\phi$ of $W$, one defines the \textit{almost character} of $G(\kk)$ associated to $\phi$ by
\begin{equation}\label{equation:almost-character}
R_\phi=\abs{W}^{-1}\sum_{w\in W}\phi(w)R_w.
\end{equation}
Note that, as $\abs{\phi(w)}\le\phi(1)$, in the case where $\phi$ is irreducible, we have that  $\alpha_w:=\phi(w)/\abs{W}$ has modulus $<1$. 

The unipotent characters of $G(\kk)$ are partitioned into families, whereby $\chi,\chi'\in \mcal{E}(G(\kk),(1))$ lie in the same family if and only if there exists $\phi\in \irr(W)$ such that both $\chi$ and $\chi'$ occur with non-zero coefficient in $R_\phi$. This partition mirrors the partition of $\irr(W)$ into families (defined in \cite[\S~4.2]{Lusztig-Book}). Specifically, any for any family $\family\subseteq\mcal{E}(G(\kk),(1))$ there exists a unique family $\mbf{f}\subseteq \irr(W)$ such that $\family=\set{\chi: \inner{R_\phi,\chi}_{G(\kk)}\ne 0\text{ for some }\phi\in \mbf{f}}$.

Recall that the irreducible characters of $W$ are in bijection with the irreducible constituents of $R_{1}=\ind_{B_1(\kk)}^{G(\kk)}(\id_{B_1(\kk)})$ (recall $B_1$ is a $\kk$-defined Borel); given $\phi\in\irr(W)$, we write $\chi_\phi$ for the corresponding constituent of $R_1$.

To each family $\family\subseteq \mcal{E}(G(\kk),(1))$ one associates a finite group $\Gamma_{\family}$ such that the family $\family$ is in bijection $\chi\mapsto m_\chi$ with a distinguished finite set $\mset(\Gamma_\family)$ (see, e.g., \cite[\S~4]{Lusztig-TypeE8} for the explicit description of $\mset(\Gamma)$, for $\Gamma$ a finite group). On $\mset(\Gamma_\family)$ one defines a pairing $\set{\cdot,\cdot}_\family$, which gives rise to a hermitian and unitary bilinear form on $\Span_\dbC\family$, also denoted $\set{\cdot,\cdot}_\family$, determined by the rule $\set{\chi,\chi'}_\family=\set{m_\chi,m_{\chi'}}_{\family}$ for all $\chi,\chi'\in\family$. Assuming $\kk$ is large enough, this pairing satisfies
\begin{equation}\label{equation:pairing-property}
\inner{\chi,R_\phi}_{G(\kk)}=\begin{cases}\set{\chi,\chi_\phi}_\family&\text{if }\chi_\phi\in\family,\\0&\text{otherwise},\end{cases}\end{equation}for any $\chi\in\family$ and $\phi\in\irr(W)$ (see \cite[Theorem~1.5]{Lusztig-Exceptional}, \cite[Theorem~5.8]{Lusztig-Symplectic} and \cite[Theorem~3.15]{Lusztig-Orthogonal}). The matrix $\mbf{A}(\family)=\left(\set{\chi,\chi'}_{\family}\right)_{\chi,\chi'\in\family}$, representing this pairing, is known as the \textit{non-abelian Fourier matrix} associated to the family $\family$. 

In particular, the description of the almost characters $R_\phi$ as a linear combination of irreducible characters is determined by the associated Fourier matrices, which are, in turn, determined by the group $\Gamma_\family$. The number of options for $\Gamma_\family$ is known to be extremely limited; specifically, $\Gamma_\family$ is necessarily either trivial, a permutation group $\mfr{S}_i$ on $i=3,4$ or $5$ elements, or a direct product of at most $\mathrm{rank} (G)$-many copies of $\mfr{S}_2$. The associated Fourier transform matrices are then computed, resulting in the following cases (see \cite[p.~110]{Lusztig-Book})
\begin{enumerate}
\item 
If $\Gamma_\family=\set{1}$, then $\mbf{A}(\family)=(1)$ and $\family$ consists of a single element  which is necessarily a constituent of $\ind_{B(\kk)}^{G(\kk)}(\id_{B(\kk)})$, hence of the form $\chi_\phi$ for $\phi\in \irr(W)$. In this situation $R_\phi=\chi_\phi$ is an irreducible character. 
\item 
If $\Gamma_\family=\mfr{S}_2^e$ then $\abs{\mset(\family)}=\abs{\family}=4^{e}$ and the family $\family$ may be ordered as $(\chi_i:i=1,\ldots,4^e)$ such that $\mbf{A}(\family)$ is represented by a block diagonal matrix consisting of $4\times 4$ blocks of the form
\begin{equation}\label{equation:fourier-matrix-4by4}
\frac{1}{2}\begin{pmatrix}1&1&1&1\\
1&1&-1&-1\\
1&-1&1&-1\\
1&-1&-1&1
\end{pmatrix};
\end{equation}see \cite[\S~4]{Lusztig-TypeE8}.
\item 
In the remaining cases, where $\family=\mfr{S}_3,\mfr{S}_4$ or $\mfr{S}_5$, the family $\family$ consists of $8,\:21$ or $39$ characters, and the corresponding Fourier transform matrices are given in \cite[pp.\ 110-113]{Lusztig-Book}. Note that in all cases, suitably ordering the elements of $\family$, the first row and column of $\mbf{A}(\family)$ consist of strictly positive rational numbers.
\end{enumerate}

\begin{proof}[Proof of \autoref{lem:unipotents-pos-in-almost-char} for split adjoint groups]
	Let $G$ be simple adjoint with root system $\Phi$ over a finite field $\kk$ and let $\chi\in\mcal{E}(G(\kk),(1))$. Let $N_\Phi$ be large enough so that \eqref{equation:pairing-property} holds in $G(\kk)$.	Let $\family\subseteq\mcal{E}(G(\kk),(1))$ be the family containing $\chi$. If $\abs{\family}=1$ then, as already mentioned, there exists $\phi\in\irr(W)$ such that $\chi=\chi_\phi=R_\phi$ and we may take $R(\chi)=\chi$ and $C_{\Phi}=\set{1}$. If $\abs{\family}\in\set{8,21,39}$, we may take $R(\chi)$ to be the almost character whose coordinate vector (with respect to the basis $\mcal{E}(G(\kk),(1))$) is given by the strictly positive column of $\mbf{A}(\family)$. In this setting, the lemma follows by taking $C_\Phi$ to the set of all positive entries of $\mbf{A}(\family)$. Finally, if $\abs{\family}=4^{e}$, by fixing an ordering $(\chi_1,\ldots,\chi_{4^e})$ of the elements of $\family$ as in item (2) above, for suitable $j=1,\ldots,4^{e-1}$ and $k=1,2,3,4$ we have that $\chi=\chi_{4(j-1)+k}$ and we may take $C_\Phi=\set{1/2}$ and $R(\chi)=\frac{1}{2}\left(\chi_{4(j-1)+1}+\cdots+4_{4(j-1)+4}\right)$. 
\end{proof}

\subsection{Non-split adjoint groups}

Having reviewed the proof in the split setting, we now explain how to derive \autoref{lem:unipotents-pos-in-almost-char} for non-split simple adjoint groups. Recall that we denote by $\tau_\frob$ the automorphism induced on $W$ by the action of the Frobenius map $\frob$. Any $\tau_\frob$ invariant character extends to a character of $\tilde{W}=T\ltimes\gen{\tau_\sigma}$, and one defines the almost characters in this setting by 
	\[R_\phi=\frac{1}{\abs{W}}\sum_{w\in W}\phi( w\cdot\tau_\frob)R_w,\]
where $R_w=\sum_{i}(-1)^i\Tr(\:\cdot\:; H^i_c(X_w,\widebar{\dbQ}_\ell))$ is defined as above (for the given non-split $\kk$-structure of $G$) and $\phi\in \irr(W)^{\tau_\phi}$. Again, as $\phi$ comes from an irreducible character of $W$, we have that $\abs{\phi(z)}<\abs{W}$ for all $z\in \tilde{W}$.

We consider the the various cases of non-split adjoint simple groups case-by-case.

\paragraph{\it Types ${}^2A_n$ and ${}^2E_6$} Assume $G$ is of type $A_n$ or $E_6$ and that the Frobenius map induces a non-trivial involution of $W$. In both cases, it is known that $\tau_\frob$ is an inner automorphism of $W$ and thus acts trivially on $\irr(W)$. In particular, there exists an obvious bijection between set of almost characters of $G(\kk)$ and of $G^{sp}(\kk)$, where $G^{sp}(\kk)$ denotes the split $\kk$-form of $G(\kk)$. In the case where $G(\kk)$ is of type ${}^2A_n$, Lusztig and Srinivasan have shown that all almost characters of $G$ are irreducible characters up to sign; see \cite[Theorem~2.2]{Lusztig-Srinivasan}. Otherwise, if $G(\kk)$ is of type ${}^2E_6$, by \cite{Lusztig-Exceptional}, there exists a bijection $\mcal{E}(G(\kk),(1))\leftrightarrow\mcal{E}(G^{sp}(\kk),(1))$, mapping mapping almost characters to almost characters and preserving the corresponding Fourier transform matrices, up to sign.

In both cases, \autoref{lem:unipotents-pos-in-almost-char} follows by applying the argument applies in the split case.

\paragraph{\it Type ${}^2D_n$} The unipotent characters of a simple adjoint group $G$ over $\kk$ of type $D_n$ and ${}^2D_n$ over $\kk$ were classified in \cite{Lusztig-Orthogonal}. In this setting, the unipotent representations of $G$ are partitioned into families of size $2^e$, for $1\le e\le n$, and the corresponding Fourier transform matrices are given by block-diagonal matrices with diagonal blocks as in \eqref{equation:fourier-matrix-4by4}. In this setting, again, one may apply the argument of the corresponding split case in order to prove \autoref{lem:unipotents-pos-in-almost-char}. \footnote{Alternatively, one may invoke \cite[Proposition~3.13]{Lusztig-Orthogonal} to obtain a more direct argument, by taking $R(\chi)$ (for $\chi\in\mcal{E}(G(\kk),(1))$ to be the virtual character denoted in loc.\ cit.\ as $R(\underline{c})$, where $\underline{c}$ is the unique (virtual) cell representation of $W$ containing an irreducible constituent whose corresponding principle series character lies in the same family as $\chi$. Note that such an argument requires a larger upper bound on modulus of coefficients of $R(\chi)$, when described as a linear combination of Deligne-Lusztig character. }

%
\paragraph{\it Type ${}^3D_4$}
The almost characters of simple adjoint groups of type ${}^3D_4$ and their description as a linear combination of unipotent characters is given in \cite[Theorem~1.18]{Lusztig-Exceptional}. In particular, assuming $\abs{\kk}$ is large enough and using the notation of loc.\ cit.\ , any unipotent character of $G(\kk)$ occurs with coefficient $1$ or $\frac{1}{2}$ in one of the almost characters $\widetilde{R}_1,\: \widetilde{R}_\epsilon,\: \widetilde{R}_\rho,\: \widetilde{R}_{\rho\otimes\epsilon}$ or $\widetilde{R}_8$.
\paragraph{\it Zusuki and Ree groups} The almost characters of the simple groups arising as fixed points of exceptional endomorphisms of a simple $\kk$-group were described by Geck and Malle, who described their decomposition as linear combinations of unipotent characters and the associated Fourier transform matrices in \cite[Theorem~5.4]{Geck-Malle} (with no restrictions on the cardinality of $\kk$). The proof of \autoref{lem:unipotents-pos-in-almost-char} in this setting follows using the same argument as in the split case.

\bibliographystyle{plain}
\bibliography{bounds}
\end{document}